\documentclass[11pt,a4paper]{article}
\usepackage{a4,amssymb,amsmath,amsthm,url,color,enumerate}
\usepackage{bezier,amsfonts,amssymb,graphicx,amsthm,url}%5,mathtools}
\usepackage[english]{babel}
\title{Independence and Matchings  in $\sigma$-hypergraphs}
\date{}
\begin{document}
\newtheorem{theorem}{Theorem}[section]
\newtheorem{definition}{Definition}[section]
\newtheorem{proposition}[theorem]{Proposition}
\newtheorem{corollary}[theorem]{Corollary}
\newtheorem{lemma}[theorem]{Lemma}

\author{Yair Caro \\ Department of Mathematics\\ University of Haifa-Oranim \\ Israel \and Josef  Lauri\\ Department of Mathematics \\ University of Malta
\\ Malta \and Christina Zarb \\Department of Mathematics \\University of Malta \\Malta }
\DeclareGraphicsExtensions{.pdf,.png,.jpg}

\maketitle
\begin{abstract}
Let $\sigma$ be a partition of the positive integer $r$. A $\sigma$-hypergraph $H=H(n,r,q|\sigma)$ is an $r$-uniform hypergraph on $nq$ vertices which are partitioned into $n$ classes $V_1, V_2, \ldots, V_n$ each containing $q$ vertices. An $r$-subset $K$ of vertices is an edge of the hypergraph if the partition of $r$ formed by the non-zero cardinalities $|K\cap V_i|, 1\leq i \leq n,$ is $\sigma$.

In earlier works we have considered colourings of the vertices of $H$ which are constrained such that any edge has at least $\alpha$ and at most $\beta$ vertices of the same colour, and we have shown that interesting results can be obtained by varying $\alpha, \beta$ and the parameters of $H$ appropriately. In this paper we continue to investigate the versatility of $\sigma$-hypergraphs by considering two classical problems: independence and matchings.

We first demonstrate an interesting link between the constrained colourings described above and the $k$-independence number of a hypergraph, that is, the largest cardinality of a subset of vertices of a hypergraph not containing $k+1$ vertices in the same edge. We also give an exact computation of the $k$-independence number of the $\sigma$-hypergraph $H$.  We then present results on maximum, and sometimes perfect, matchings in $H$. These results often depend on divisibility relations between the parameters of $H$ and on the highest common factor of the parts of $\sigma$. 
\end{abstract}

\section{Introduction}

Let $V=\{v_1,v_2,...,v_n\}$ be a finite set, and let $E=\{E_1,E_2,...,E_m\}$ be a family of subsets of $X$.  The pair $H=(X,E)$ is called a \emph{hypergraph} with vertex-set $V(H)=V$, and with edge-set $E(H)=E$.  When all the subsets are of the same size $r$, we say that $H$ is an \emph{r-uniform hypergraph}.   A $\sigma$-hypergraph $H= H(n,r,q$ $\mid$ $\sigma$), where $\sigma$ is a partition of $r$,  is an $r$-uniform hypergraph having $nq$ vertices partitioned into $ n$ \emph{classes} of $q$ vertices each.  If the classes are denoted by $V_1$, $V_2$,...,$V_n$, then a subset $K$ of $V(H)$ of size $r$ is an edge if the partition of $r$ formed by the non-zero cardinalities $ \mid$ $K$ $\cap$ $V_i$ $\mid$, $ 1 \leq i \leq n$, is $\sigma$. The non-empty intersections $K$ $\cap$ $V_i$ are called the parts of $K$, and $s(\sigma)$ denotes the number of parts.   The first two authors introduced $\sigma$-hypergraphs in \cite{CaroLauri14}, studying a particular instance of Voloshin colourings of these r-uniform hypergraphs (see \cite{voloshin02} for a detailed study of these colourings).  The chromatic spectra and other properties were further studied in \cite{CLZ2,CLZ1}.  In this paper, we turn our attention to some classic properties of graphs and hypergraphs, and investigate these properties for $\sigma$-hypergraphs.

A set of vertices of a hypergraph is said to be \emph{independent} if it contains no edges.   The \emph{independence number} $\alpha(H)$ of a hypergraph $H$ is the size of a largest independent set of vertices of $H$.   The problem of finding the maximum independent set in a graph, and also in hypergraphs, is a well-known $NP$-hard problem in Graph Theory (as described in \cite{garey1979computers}).  

We also consider the following variation of independence: a set of vertices $S$  in an $r$-uniform hypergraph $H$ is said to be \emph{$k$-independent},  for $1 \leq k  \leq r-1$, if for every edge $E \in E(H)$, $|E \cap  S | \leq k$.  The largest cardinality of a $k$-independent set is denoted $\alpha_{k}(H)$.  We observe that $(r-1)$-independence is  the classical notion of independence defined above and hence $\alpha(H)=\alpha_{r-1}(H)$, while $1$-independence is sometimes called  \emph{strong independence}.  To prevent confusion, we stress here that our notion of $k$-independence in $r$-uniform hypergraphs has no relation with the notion of the $k$-independent number in graphs (see \cite{CaroHan13,ChellaliFHV12,HansbergPep13}).

We then consider matchings in $\sigma$-hypergraphs. Given an $r$-uniform hypergraph $H$, a \emph{matching} is a set of pairwise vertex-disjoint edges $M \subset E(H)$.  A \emph{perfect matching} is a matching which covers all vertices of $H$.  Graphs which contain a perfect matching are characterised by a theorem of Tutte (as cited in \cite{plummermatching}), but deciding whether an $r$-uniform hypergraph contains a perfect matching is an NP-complete problem for $r \geq 3$, as discussed in \cite{huang2012size}.  As in \cite{huang2012size}, we denote the size of the largest matching in an $r$-uniform hypergraph $H$ by $\nu(H)$.  If $H$ has a perfect matching, then $\nu(H)=\frac{|V(H)|}{r}$, so clearly $\nu(H)\leq\frac{|V(H)|}{r}$.  

This paper is organised as follows.  We first consider independence and $k$-independence in $\sigma$-hypergraphs.  We start by looking at an interesting link between $k$-independence and the upper  and lower chromatic numbers $\overline{\chi}_{\alpha,\beta}$  and $\chi_{\alpha,\beta}$ respectively, for a constrained colouring of a $\sigma$-hypergraph, which is studied extensively in \cite{bujtas2007colour,bujtastuz09,bujtas2009color,bujtas2010color,bujtastuz13,bujtasv2011color,CLZ1}.  A constrained colouring, or $t$-$(\alpha,\beta)$-colouring  is a type of hypergraph colouring of the vertices using $t$ colours such that each edge has at least $\alpha$ and at most $\beta$ colours appearing on its vertices.  This type of colouring was first defined in \cite{bujtastuz09}.   The \emph{lower chromatic number} $\chi_{\alpha,\beta}$ is defined as the least number $k$ for which $H$ has a $k$-$(\alpha,\beta)$-colouring.  Similarly, the \emph{upper chromatic number} $\overline{\chi}_{\alpha,\beta}$ is the largest $k$ for which $H$ has a $k$-$(\alpha,\beta)$-colouring. This serves as further motivation to study these parameters further.  We then present an exact computation for the $k$-independence number in $\sigma$-hypergraphs.  We then move on to matchings in $\sigma$-hypergraphs and present tight bounds for $\nu$, as well as conditions for the existence of and constructions of perfect matchings.  We conclude with some further considerations and open questions for maximum matchings.

\section{Independent sets in $\sigma$-hypergraphs}

In this section we develop some lemmas and tools that lead to Theorem \ref{alpha_k}, in which we present a complete, effective and easily computable formula for the $k$-independence number $\alpha_k$ for $\sigma$-hypergraphs.  Some examples are given after this theorem is proved.
 
An important, though simple, link between the $k$-independence number of hypergraphs and the upper and lower  $(\alpha,\beta)$-chromatic number, $\overline{\chi}_{\alpha,\beta}$ and $\chi_{\alpha,\beta}$ respectively, serves as our starting result and motivation, as it connects the current work to  previous work in \cite{bujtastuz09,CaroLauri14,CLZ2,CLZ1},  which concentrated on $(\alpha,\beta)$-colourings of $\sigma$-hypergraphs.  This is similar in concept to the $C$-stability number as an upper bound for  $\overline{\chi}$  in mixed hypergraphs \cite{tuza2008problems,voloshin02}.

 We first prove a simple lemma to be used in this Proposition.

\begin{lemma} \label{sums}
Let $1 \leq x_1 \leq x_2 \leq \ldots \leq x_k$ be positive integers with \[\sum_{j=1}^{k}x_j \leq q.\]  Then for every $t$, $1 \leq t \leq k$, \[\sum_{j=1}^t x_j \leq \frac{tq}{k} .\]
\end{lemma}

\begin{proof}
Clearly 
\[\sum_{j=1}^t x_j = \sum_{j=1}^k x_j  - \sum_{j=t+1}^k x_j  \leq q- \sum_{j=t+1}^k x_j \]\[\leq  q - (k-t)x_{t+1} \leq q - (k-t)x_{t} \leq  q - \left(\frac{k-t}{t} \right )\sum_{j=1}^t x_j\]

Hence \[t\sum_{j=1}^t x_j  + (k-t)\sum_{j=1}^t x_j = k\sum_{j=1}^t x_j  \leq tq,\] and therefore \[ \sum_{j=1}^t x_j  \leq \frac{tq}{k}.\]

\end{proof}

\begin{proposition} \label{linkcolind}
Let $H$ be an $r$-uniform hypergraph.  Then  \[\alpha_{\beta}(H) \geq \overline{\chi}_{\alpha,\beta}(H) \geq \chi_{\alpha,\beta} \geq \frac{ (\alpha-1)|V(H)|}{\alpha(H)}.\]  
\end{proposition}

\begin{proof}
Consider a colouring of $H$ using $\alpha_{\beta}(H)+1$ colours.  Consider a set $D$ of vertices of $H$ such that each of the $\alpha_{\beta}(H)+1$ colours used is represented in $D$.  Then $|D| = \alpha_{\beta}(H)+1$ , and by the definition of the $\beta$-independence number, there exists an edge $E$ such that $|E \cap D| \geq \beta+1$, hence $E$ contains at least $\beta+1$ distinct colours and therefore this is not a valid $(\alpha,\beta)$-colouring.  Thus the number of colours which can be used in an $(\alpha,\beta)$-colouring of $H$ is at most $\alpha_{\beta}(H)$, that is $\alpha_{\beta}(H) \geq \overline{\chi}_{\alpha,\beta}(H)  \geq \chi_{\alpha,\beta}$.

For the last part, let $A_1,A_2,\ldots,A_p$, where $p= \chi_{\alpha,\beta}(H)$, be a partition of $V(H)$ into monochromatic colour classes with $|A_1| \leq |A_2| \leq \ldots \leq |A_p|$.  Clearly, the union of any $\alpha-1$ colour classes form an independent set, otherwise there would be an edge with at most $\alpha-1$ colours, a contradiction.  Hence such a union has cardinality at most $\alpha(H)$. 

Now let $\chi_{\alpha,\beta}(H)=p=m(\alpha-1)+z$, where $0 \leq z \leq \alpha-2$.  Then we have $m$ sets of $\alpha-1$ colour classes, and each such set has cardinality at most $\alpha(H)$, and hence the sum of the cardinality of these classes is at most $m\alpha(H)$.  The remaining $z$ colour classes have total cardinality at most $\frac{z \alpha(H)}{\alpha-1}$, by Lemma \ref{sums}, using $q=\alpha(H)$, $k=\alpha-1$ and $z=t$, and noticing that these $z$ classes are the smallest $z$ classes.

Thus 

\[|V(H)|=\sum_{j=1}^{p} |A_j| \leq  m\alpha(H)+\frac{z \alpha(H)}{\alpha-1} =\frac{\alpha(H)}{\alpha-1}(m(\alpha-1)+z)=\frac{\alpha(H)p}{\alpha-1},\]

and hence \[p=\chi_{\alpha,\beta}(H) \geq \frac{(\alpha-1)|V(H)|}{\alpha(H)}.\]

\end{proof}

 In particular, the above Proposition gives a necessary condition for the existence of an $(\alpha,\beta)$-colouring of an $r$-uniform hypergraph.

\begin{corollary}
Let $H$ be an $r$-uniform hypergraph.  If $|V(H)| > \frac{\alpha(H)\alpha_{\beta}(H)}{\alpha-1}$, then no $(\alpha,\beta)$-colouring of $H$ exists.
\end{corollary}

\begin{proof}
The lower and upper bounds $\chi_{\alpha,\beta}$ and $\overline{\chi}_{\alpha,\beta}$  must lie between $\frac{(\alpha-1)|V(H)|}{\alpha(H)}$ and $\alpha_{\beta}(H)$ respectively.  Hence if \[\frac{(\alpha-1)|V(H)|}{\alpha(H)} > \alpha_{\beta}(H), \mbox{ that is } |V(H)| > \frac{\alpha(H)\alpha_{\beta}(H)}{\alpha-1},\] then no $(\alpha,\beta)$-colouring of $H$ exists.
\end{proof}.

We now start by proving a  lemma inspired by the well-known rearrangement inequality  (as cited in \cite{hardy1952inequalities}).

\begin{lemma} \label{setinter}
 
Let $H=H(n,r,q \mid \sigma)$ be a $\sigma$-hypergraph with $\sigma=(a_1,\ldots,a_s)$  and $a_1 \geq a_2 \geq \ldots \geq a_s \geq 1$.  Let $B$ be a subset of $V(H)$ and let $B_i= B  \cap V_i$ with $|B_i|= b_i$ where $b_1 \geq b_2 \geq \ldots \geq b_n$.   
Let $E^*$ be the edge with part $A_ i $ of cardinality $a_i$ of $\sigma$  located in $V_i$, such that $A_i \subset B_i$ or $B_i \subset A_i$  when $a_i < b_i$ or $a_i \geq b_i$ respectively.  Then $|E^* \cap B|=\max\{|E \cap B|: E \in E(H)\}$.
 \end{lemma}
\begin{proof}
 
Consider the case when an edge  $E$ has some part $A_ i \not \in V_1,\ldots,V_s$. Then some class $V_ j$, $j=1 \ldots s$, contains no element from $E$.  The part $A_i$ is located in some $V_ z$ where $z > s$ , hence $b_ j \geq b_s \geq b_z$.
 
Let  $E^{J}$ be the edge with the part $A_i$  located in $V_ j$ (with maximum intersection with $B_ j$),  and with all other parts as in edge $E$. Then
\begin{eqnarray*}
| E^{J} \cap B | &=&  | (E^{J}\setminus A_i) \cap B)| + |B_ j \cap  A_i | \\
 &=& | (E \setminus A_i ) \cap B | + |B_ j  \cap A_ i | \\
&\geq& (E \setminus A_i ) \cap  B | + |B_z \cap  A_ i | 
\end{eqnarray*}

and hence $| E^{J} \cap B | \geq  | E \cap B|$  since $b_ j \geq b_ z$.

Therefore, we need only consider edges whose parts are located in $V_1,\ldots,V_s$,  with maximum intersection with the parts of $B$.

Now suppose there exists $i$, $1 \leq i \leq s$, such that $A_i$ is not located in $V_i$ and let $i$ be the smallest such value. Then some part $A_ j$  is located in $V_i$, where $ i < j$,  because for $j < i$, $A_ j$ is located in $V_ j$ by definition of the smallest $i$.   Also $A_ i$ is located in some $V_ z$ for $z \geq i +1$.  Observe first that  $a_ i \geq a _ j$ and $b_ i \geq b_z$.

Let us replace the position of the parts $A_ i$ and $A_ j$ to get an edge $E^Z$ so that $A_ i$ is now located in $V_ i$ with maximum intersection with $B_ i$ and $A_ j$ is now located in $V_z$  with maximum intersection with $B_z$.  We need to consider the following cases:

\begin{enumerate}[i.]
\item{Consider the case when $b_i \geq a_i$ and $b_z \geq a_j$.  Then clearly \[|A_ i \cap B_i | + |A_ j \cap B_z | \geq  |A_ j \cap B_i | + |A_ i \cap B_z | \mbox{ and hence } | E^Z \cap  B | \geq | E \cap  B|\]}
\item{ Consider the case when $b_i \geq a_i$ and  $b_z  \leq a_ j$.  Then clearly  \[|A_ j \cap B_z | = |A_ i \cap B_z | \mbox{ and } |A_ j \cap B_i | \leq |A_ i \cap B_i |,\]  and once again  \[|A_ i \cap B_i | + |A_ j \cap B_z | \geq  |A_ j \cap B_i | + | A_ i \cap B_z | \mbox{ hence } | E^ Z \cap B | \geq | E \cap  B|.\]}
\item{Consider the case when $b_i < a_i$ and $b_z \geq a_ j$.  Then  \[| A_i \cap B_i | + |A_ j \cap B_z | = b_ i + a_ j  \mbox { while } | A_ i \cap B_z |+ |A_ j \cap B_i | =  \min\{ a_i ,b_z\} + \min\{ a_j ,b_i\}=  b_z+ a_j,\] since $a_ i > b_ i  \geq b_ z \geq a_ j$. Hence $ | E^ Z \cap B | \geq | E \cap  B|$.}
\item{  Finally, consider the case when $ b_i < a_i$ and $b_z <  a_ j$.  Then  \[|A_ i \cap B_i | + |A_ j \cap  B_z | = b_ i + b_z  \mbox{ while } | A_ i \cap B_z |+ |A_ j \cap B_i | =  \min\{ a_i ,b_z\} + min\{ a_j ,b_i\}     \leq   b_z + b_i,\] again giving $ | E^ Z \cap B | \geq | E \cap  B|$.}
\end{enumerate}

Hence we can relocate the part $A_i$ which was located in $V_ j$ (for some $j >i$), and put it in $V_i$.    Then the smallest $i$ for which $A_i$ is not optimally located in $V_i$ has increased, and we can repeat the process until $A_i$ is located in $V_i$ for all $ 1 \leq i \leq s$, giving the edge $E^*$ as required.
\end{proof}

Consider the $\sigma$-hypergraph $H=H(n,r,q \mid \sigma)$, with $\sigma=(a_1,\ldots,a_s)$, and $a_1 \geq a_2 \geq \ldots \geq a_s  \geq 1$. 

Let $1 \leq k \leq r-1$. Consider the sequence $B=(b_1,b_2,\ldots,b_n)$, where $b_1 \geq b_2 \geq \ldots \geq b_s \geq \ldots \geq b_n$, and for $j \geq s=s(\sigma)$, $b_j=b_s$, and $q \geq \max\{a_1,b_1\}$.  Then this sequence is said to be \emph{$(q,k,\sigma)$-feasible} if \[\sum_{i=1}^{s} \min\{a_i,b_i\} = k.\]  

\begin{lemma} \label{sigma_feas}

Let $B=(b_1,b_2,\ldots,b_n)$ be a $(q,k,\sigma)$-feasible sequence.  Then there exists $t=t(B) \geq 1$ such that for $j<t \leq s$, $b_j \geq a_j$ while $b_t<a_t$.
\end{lemma}

\begin{proof}
 
If such $t=t(B)$ does not exist, then \[\sum_{i=1}^{s} \min\{a_i,b_i\} = \sum_{i=1}^{s}a_i = r >k,\] a contradiction.
\end{proof}

\begin{lemma}
For a given partition $\sigma=(a_1,\ldots,a_s)$ where \[\sum_{i=1}^{s}a_i=r,\]  there exists at least one $(q,k,\sigma)$-feasible sequence for all values of $k$ such that $1 \leq k \leq r-1$.
\end{lemma}

\begin{proof}
Consider the sequence where $b_i=a_i$ for $1 \leq i \leq s$, and $b_j=b_s=a_s$ for $j > s$.

Now consider the sequence obtained by setting $b_j=b_s-1$ for $j \geq s$.  Then \[ \sum_{i=1}^{s} \min\{a_i,b_i\}=r-1,\] giving an $(q,r-1,\sigma)$-feasible sequence.

Now let us assume $ 1 \leq t<r-1$, and let $(b_1,\ldots,b_s,\ldots,b_n)$ be a $(q,t,\sigma)$-feasible sequence.  If $b_s>0$, then the sequence obtained by subtracting 1 from $b_i$, for $ s \leq i \leq n$ is a $(q,t-1,\sigma)$-feasible sequence.  Otherwise for some value of $i$, $1 \leq i < s$, $b_i>0$.  Let $j$ be the largest such index.  Then the sequence $(b_1,b_2,\ldots,b_j-1,0,\ldots,0)$ is a $(q,t-1,\sigma)$-feasible sequence.

Hence, by induction, there exists a $(q,k,\sigma)$-feasible sequence for all $ 1 \leq k \leq r-1$.

\end{proof}

A $(q, k,\sigma)$-feasible sequence $B^* = (b^*_1,\ldots,b^*_n )$ is said to \emph{dominate} the $(q, k,\sigma)$-feasible sequence $B = ( b_1,\ldots,b_n)$  if  $b^*_i \geq b_i$ for all $1 \leq i \leq n$ and \[\sum_{i=1}^{n} b^*_i >  \sum_{i=1}^{n} b_i.\]
 
A  $(q,k,\sigma)$-feasible sequence $B$ is \emph{maximal} if it is not dominated by any $(q,k,\sigma)$-feasible sequence $B^* $.

Let $M(q,k,\sigma)$ be the set of all maximal $(q,k,\sigma)$-feasible sequences.

\begin{lemma} \label{max_feasible}
Let  $\sigma =(a_1,\ldots,a_s)$  with $a_1 \geq a_2 \geq \ldots \geq a_s \geq 1$  be a partition of $r$.  Let $B  = (b_1,\ldots,b_n)$ be a $(q,k,\sigma)$-feasible sequence.  If $B$  is not a maximal $(q,k,\sigma)$-feasible sequence then it  can be extended to a maximal $(q,k,\sigma)$-feasible sequence $B^* =(b^*_1,\ldots,b^*_n)$  which dominates $B$.
\end{lemma}

\begin{proof}
Suppose $B$ is not a maximal $(q,k,\sigma)$-feasible sequence.  
 
Then by definition of maximality there is a $(q,k,\sigma)$-feasible sequence $B^*$ that dominates it, which means that for some $i$, $1 \leq i \leq s$ , we can replace $b_ i$ by $b_i^*= b_i +1$, preserving \[\sum_{j=1,j \not =i}^{s} \min \{a_ j,b_ j \} + \min\{ a_i,b^*_i\}=k\]   and monotonicity, and keeping $q \geq \max\{a_1,b_1\}$ (or $q \geq \max\{a_1,b_1^*\}$ if $i=1$), and in case $i = s$,  we also replace $b_j$ by $b^*_ j = b_s +1$ for $j\geq s$. This gives the sequence $B^*=(b^*_1,\ldots,b^*_n)$ which is $(q,k,\sigma)$-feasible and is such that $b^*_i \geq b_i$ for all $1 \leq i \leq n$, $q \leq \max \{a_1,b_1^*\}$ and \[\sum_{i=1}^{n}b^*_i > \sum_{i=1}^{n}b_i\].   
 
This process can be repeated but must terminate, since each time we increase \[\sum_{i=1}^{n} b_i^*\] by at least 1, and trivially $b_i \leq q$ for $1 \leq i \leq n$,  hence \[\sum_{i=1}^{n} b_i^* \leq qn.\]
 
Therefore the process gives a maximal $(q,k,\sigma)$-feasible sequence $B^*$, such that \[\sum_{i=1}^{n}b^*_i > \sum_{i=1}^{n}b_i\]. 
\end{proof}
 
We now present our main result for $k$-independence in $\sigma$-hypergraphs.

\begin{theorem} \label{alpha_k}
Consider $H=H(n,r,q \mid \sigma)$ with $\sigma = (a_1,\ldots,a_s)$  where $a_1 \geq a_2 \geq \ldots \geq a_s  \geq 1$.   Consider $1 \leq k \leq r-1$.  Then  \[\alpha_{k}(H) = \max \{  q(t(B)-1) + \sum_{i=t}^{s}b_i+ (n-s)b_s  : B \in M(q,k,\sigma)\}.\]
\end{theorem}

\begin{proof}

Let $B$ be  $k$-independent set of maximum cardinality in $H$. Let $ B _ i   =  B \cap V_ i $  and let $b_i=|B_i|$, and  we assume, without loss of generality,  that $q \geq b_1 \geq b_2 \geq \ldots\geq b_n \geq 0$.   By Lemma \ref{setinter} we can consider just the edge $E^*$ and look at its intersection with $B$.

Let us consider $b_1 ,b_2 ,\ldots,b_n$.  For $j \geq s=s(\sigma)$ we may take $b_j=b_s$, otherwise $B$ is not maximal since the maximum intersection of an edge with any $s$ classes is not larger than the intersection of $E^*$ with the first $s$ classes, by Lemma \ref{setinter}.  Since we are considering $E^*$, we assume the $A_i$ is located in $V_i$ with optimal (maximum) intersection with $B_i$.

We observe that \[|E^* \cap  B| =  \sum_{i=1}^{s} |A_i \cap  B_i |  = \sum_{i=1}^{s} \min\{ a_i,b_i\} \leq k\]   since $B$ is $k$-independent.  Hence, for some integer $t$, $t \leq s$  we have  $b_{t-1} \geq a_{t-1}$ but $b_t  < a_ t $, by Lemma \ref{sigma_feas}.
%because otherwise \[\sum_{i=1}^{i=s} \min\{ a_i , b_i\}  =  \sum_{i=1}^{s} a_ i = r > k.\]

Since $B$ is of maximum cardinality then \[|E^* \cap B|=\sum_{i=1}^{s} |A_i \cap B_i|=k,\] otherwise, by Lemma \ref{setinter}, for all edges $E \in E(H)$, $|E \cap B| \leq |E^* \cap B|<k$ and we can add a vertex to $B$ in $B_t$ where $t$ is the smallest index for which $b_i<a_i$, to get a set $B^{*}$, which is still monotonic since either $t \geq 2$ and $b_{t-1} \geq a_{t-1} \geq a_t > b_t$ or $t=1$ in which case $b_1$ is the largest element anyway.  Therefore $B^*$ is such that $|E^* \cap B^*| = |E^* \cap B|+1 \leq k$, contradicting the maximality of $B$.

Therefore we can conclude that $(b_1,b_2,\ldots,b_n)$ is a $(q,k,\sigma)$-feasible sequence.  We may assume the $B$ is infact a maximal $(q,k,\sigma)$-feasible sequence, otherwise by Lemma \ref{max_feasible}, $B$ may be extended to a maximal $(q,k,\sigma)$-feasible sequence $B^*$ since we are assuming $q \geq \max \{a_1,b_1\}$.  This gives a $k$-independent set with cardinality greater than $|B|$, contradicting the maximality of the $k$-independent set $B$.

Now for every $j < t  =t(B)$ (as defined in Lemma \ref{sigma_feas}), $ \min\{ a_ j ,b_j\} = a_ j$.  By the maximality of $B$, for $j<t$, $b_j$ and $q$ are equal, because otherwise we can add vertices to $B_j$  to get $B^*$ with $|E^* \cap B^*| = |E^* \cap B|=k$ and $|B^*| \geq |B|$, a contradiction to the maximality of $B$.  Then \[|B|=(t-1)q + \sum_{i=t}^{s}b_i + (n-s)b_s,\] which implies \[\alpha_k(H)=\max \{ q(t(B)-1)+ \sum_{i=t}^{s}b_i + (n-s)b_s  : B \in  M(q,k,\sigma)\}.\]
\end{proof}

Note : One can observe that the computation of $\alpha_k(H)$ depends only on the structure of $\sigma$ since the number of created maximal $(q,k,\sigma)$-feasible sequences as well as the number of linear inequalities to be solved depend only  on $\sigma$, and hence it is independent of the number of vertices and edges in $H$.  Therefore for fixed $r$ this is done in $O(1)$ time.
\begin{corollary}
Given $H=H(n,r,q \mid \sigma)$, then the independence number of $H$ is \[\alpha(H) = \max\{ (j-1)q+ (a_j-1)(n-j+1): j=1,\ldots,s=s(\sigma)\}.\]
\end{corollary}

\begin{proof}
It is clear that for an $r$-uniform hypergraph, $\alpha(H)=\alpha_{r-1}(H)$.  Therefore by Theorem \ref{alpha_k}, \[ \alpha(H) =\max \{ q(t(B)-1)+\sum_{i=t}^{s}b_i+ (n-s)b_s  : B \in M(q,r-1,\sigma)\}.\]  Now a $(q,r-1,\sigma)$-feasible sequence $(b_1,b_2,\ldots,b_n)$ is such that \[\sum_{i=1}^{s} \min\{a_i,b_i\} =r-1.\]  

A maximal $(q,r-1,\sigma)$-feasible sequence must be of the form \[(b_1,b_2,\ldots,b_n)=(q,q,\ldots,q,a_j-1,a_j-1,\ldots,a_j-1)\] since:
 \begin{enumerate}
\item{If $b_j \leq a_j -2$, for some $ j$,  $1 \leq j \leq s$ ,  then \[\sum_{i=1}^{s} \min\{ a_i,b_i\} \leq r-2,\] contradicting the fact proved in Theorem \ref{alpha_k} that \[|E^* \cap B|  =\sum_{i=1}^{i=s} \min\{ a_i,b_i\}   = r-1.\]}
 \item{If $b_i = a_i-1$ and $b_j  = a_j -1$ for $1 \leq i < j \leq s$  then again \[\sum_{i=1}^{s} \min\{ a_i,b_i\} \leq r-2,\] a contradiction.}
 \end{enumerate}

Hence, for precisely one index $j$, $ b_ j = a_j - 1$ and for all other indices, $b_ i =q$ if $ i < j$ while $b_i  \leq a_j -1$ if $i >j$.  But then \[\sum_{i=1}^{s} \min\{ a_i,b_i\} \leq  a_1+,\ldots,+a_{j-1}+ (a_j - 1)(s-j+1) \leq r-1,\] and equality holds if and only if $a_{k} =  a_j  -1$ for $j \leq k \leq s$, for otherwise the sum is at most $r-2$.  

So, all maximal $(q,r-1,\sigma)$-feasible sequences must have the form 
$(b_1,\ldots,b_n)$ $ = (q,q,\ldots,q,a_j-1,a_j-1,\ldots,a_j-1)$ for some $j$ , $1 \leq j \leq s$.  Note, however, that not every sequence of this form is in fact a maximal $(q,r-1, \sigma)$-feasible sequence. So this form is necessary but not sufficient  for a maximal $(q,r-1, \sigma)$-feasible sequence.  Therefore, 
\[\alpha(H)=\max\{q(j-1) + (a_j-1)(n-j+1): j=1,\ldots,s\},\] as stated.
\end{proof}

\bigskip
Let us look at an example:  let $H=H(n,9,q \mid \sigma)$ where $\sigma=(4,3,2)$.  Let us consider $\alpha_k(H)$ for $k=6,7,8$.  
\medskip

\noindent \textbf{k=8}

We compute $\alpha_8(H)=\alpha(H)$.  Then the maximal $(q,8,\sigma)$-feasible sequences are:
\begin{enumerate}
\item{$(3,3,3,\ldots,3)$ when $t=1$, giving \[\sum_{i=1}^{n}b_i=3n\]}
\item{$(q,2,2,\ldots,2)$ when $t=2$, giving \[\sum_{i=1}^{n}b_i=q+2(n-1)\]}
\item{$(q,q,1,\ldots,1)$ when $t=3$, giving \[\sum_{i=1}^{n}b_i=2q + n-2\]}
\end{enumerate}

Hence $\alpha(H) = \max\{3n,q+2n-2,2q+n-2\}$.  If $n \geq q$ then $\alpha(H)=3n$.  For $q >n$, $2q+n-2 > q+2n-2$ and hence in this case, $\alpha(H)=2q+n-2$.
\medskip

\noindent \textbf{k=7}

We now consider $\alpha_7(H)$.  Then the maximal $(q,7,\sigma)$-feasible sequences are:
\begin{enumerate}
\item{$(3,2,2,\ldots,2)$ when $t=1$, giving \[\sum_{i=1}^{n}b_i=3+2(n-1) = 2n+1\]}
\item{$(q,2,1,1,\ldots,1)$ when $t=2$, giving \[\sum_{i=1}^{n}b_i=q+2+n-2=q+n\]}
\item{$(q,q,0,\ldots,0)$ when $t=3$, giving \[\sum_{i=1}^{n}b_i=2q\]}
\end{enumerate}

Hence, if $n \geq q$, $\alpha_7(H)=2n+1$ while if $n<q$, $\alpha_7(H)=2q$.

\medskip

\noindent \textbf{k=6}

We now consider $\alpha_6(H)$.  Then the maximal $(q,6,\sigma)$-feasible sequences are:
\begin{enumerate}
\item{$(3,3,0,\ldots,0)$ or $(3,2,1,\ldots,1)$ or $(2,2,\ldots,2)$ when $t=1$ giving \[\sum_{i=1}^{n}b_i=6,  \sum_{i=1}^{n}b_i=3+2 + n-2 = n+3, \mbox{ and } \sum_{i=1}^{n}b_i=2n\] respectively.  Since $n \geq 3$, the maximum is $2n$.}
\item{$(q,2,0,0,\ldots,0)$ or $(q,1,1,\ldots,1)$ when $t=2$ giving\[\sum_{i=1}^{n}b_i=q+2 \mbox{ and } \sum_{i=1}^{n}b_i=q+n-1\] respectively.  Again, since $n \geq 3$, the maximum is $q+n-1$. }
\item{none when $t=3$ }
\end{enumerate}

Hence, if $n \geq q-1$, $\alpha_6(H)=2n$ while if $n <q-1$, $\alpha_6(H)=q+n-1$.

\section{Matchings and  $\sigma$-hypergraphs}

We now consider matchings in $\sigma$-hypergraphs.  For the purpose of this section, we need to give more structure to the vertices of the hypergraph $H=H(n,r,q \mid \sigma)$ with $\sigma=(a_1,a_2,\ldots,a_s)$, and $a_1 \geq a_2 \geq \ldots \geq a_s$.   The classes making up the vertex set are ordered as $V_1,V_2,\ldots,V_n$ and, within each $V_i$, the vertices are ordered as $v_{1,i},v_{2,i},\ldots, v_{q,i}$.  We visualise the vertex set $V(H)$ as  a $q \times n$ grid whose first row is $v_{1,1},v_{1,2},\ldots,v_{1,n}$.  We sometimes refer to the vertices $v_{1,i},v_{2,i},\ldots, v_{k,i}$ as the top $k$ vertices of the class $V_i$, and to $v_{q-k+1,i},v_{q-k+2,i},\ldots, v_{q,i}$ as the bottom $k$ vertices of $V_i$.  The vertices $v_{k,i}$ and $v_{k+1,i}$ are said to be consecutive in $V_i$.  The class $V_1$ is called the first class of vertices, and $V_n$ is the last class; $V_i$ and $V_{i+1}$ are said to be consecutive classes.  A set of vertices contained in $h$ consecutive rows and $k$ consecutive classes of $V(H)$ is said to be an $h \times k$ subgrid of $V(H)$.

A matching $M_c$ is said to be in \emph{canonical form} if the parts $a_1,a_2,\ldots,a_s$ of any edge $E$ in $M_c$ are in consecutive classes, and if, within each class, the $a_j$ vertices in $E$ coming from that class are consecutive.  Any vertices not in the matching are all consecutive at the top, or bottom, of each respective class.  It is easy to see that a maximum matching of $H$ can be rearranged into one in canonical form.

\begin{lemma} \label{consecutive}
Let $H=H(n,r,q \mid \sigma)$ be a $\sigma$-hypergraph with $\sigma=(a_1,a_2,\ldots,a_s)$, $a_1 \geq a_2 \geq \ldots \geq a_s$.  Let $M$ be a maximum matching in $H$.  Then $M$ can be changed into a matching $M_c$ in canonical form.
\begin{enumerate}
\item{For every edge $E \in M_c$, the vertices in every part $a_i$ of $E$ are consecutive in their respective class.}
\item{The unmatched vertices are all consecutive at the top or bottom of each respective class.}
\end{enumerate}
\end{lemma}

\begin{proof}
Consider M, a maximum matching -- consider the vertices in the part $a_j$ in an edge $E \in M$ taken from the class $V_t$.  The vertices in this class can be reordered by some permutation so the the vertices in $a_j$ are consecutive in the class $V_t$.  This can be applied to every part $a_i$ in $E$ and creates a new edge $E_c$ which can replace edge $E$ in the matching.  This process can be repeated for every edge in $M$, without any effect on the already created new edge, to create a new matching $M_c$ in which the vertices of every part in every edge are consecutive in their respective class.

In a similar way, the unmatched vertices in any class can be rearranged so that they are consecutive in their respective class, and are the top or bottom vertices in this class.
\end{proof}

We will use this well-known Theorem by Frobenius in several places, and thus we state it here:

\begin{theorem} \label{Frobenius}
Let $a_1,a_2 $ be positive integers with $gcd(a_1,a_2)=1$.  Then for  $n \geq (a_1-1)(a_2-1)$, there are nonnegative integers $x$ and $y$ such that  $x a_1 +y a_2=n$.
\end{theorem}

\subsection{Divisibility Conditions}

In this section we look at divisibility conditions between certain parameters of a $\sigma$-hypergraph which imply the existence of certain types of matchings.   We start off with a result which gives a simple sufficient condition for the existence of a perfect matching in a $\sigma$-hypergraph.

\begin{lemma} \label{perfect1}
Consider $H=H(n,r,q \mid \sigma)$, where $\sigma=(a_1,\ldots,a_s)$, $n \geq s$ and $ q \geq r$. If $r$ $|$  $q$, then $H$ has a perfect matching.
\end{lemma}

\begin{proof}
It is clear that we need only show that the top $r \times n$ grid of vertices of $H$ afford a perfect matching.  Therefore consider only the top $r$ vertices in each of the classes $V_1,V_2,\ldots,V_n$. Let each column of $r$ vertices be partitioned into $s$ consecutive parts of sizes $a_1,a_2,\ldots,a_s$.  The part $a_i$ in $V_j$ will be referred to as the $i^{th}$ part in $V_j$.   The edge  $E_1$ is formed by taking the top  $a_1$ vertices from $V_1$, the second part of size $a_2$ from $V_2$  and so on, ``in diagonal fashion".  This repeated for $E_2$ by ``shift one class to the right", taking the top $a_1$ vertices from $V_2$, the second part from $V_3$ etc.  In general, the edge $E_j$, $1 \leq j \leq n$, takes the first part from $V_j$, the second part from $V_{j+1}$ and in general the $k^{th}$ part from $V_{j+k-1}$, for $1 \leq k \leq s$, with addition modulo $n$.  This gives a perfect matching of the top $r \times n$ grid consisting of $n$ edges.
\end{proof}

In contrast with the above result we next show that certain $\sigma$-hypergraphs do not have a perfect matching, and that in a maximum matching there may be many unmatched vertices.  We define $gcd(\sigma) = gcd( a_1,\ldots,a_s)$  for $\sigma = ( a_1,\ldots,a_s)$.  

\begin{lemma} \label{not_r_good}
Let $H=H(n,r,q \mid \sigma)$, where $\sigma=(a_1,\ldots,a_s)$, $n \geq s$ and $ q \geq r$.  Suppose $gcd(\sigma)=d \geq 2$, and $q=t\pmod d$ where $1 \leq t \leq d-1$. Then in a maximum matching of $H$, there are at least $tn$ vertices left unmatched.  Hence $\nu(H) \leq \frac{n(q-t)}{r}$.
\end{lemma}
\begin{proof}
Every edge in a maximum matching has all its parts divisible by $d \geq 2$. So the parts of every edge in a maximum matching cover $0\pmod d$ vertices in each class, and hence in each class there are at least $t$ vertices left unmatched.  As there are $n$ classes we have at least $nt$ vertices unmatched. 
\end{proof}

We now present a result which, in the next section, will  allow us to ``expand" a maximum matching in a $\sigma$-hypergraph to one in another $\sigma$-hypergraph with more vertices.

\begin{lemma} \label{expand}
Let $H=H(n,r,q \mid \sigma)$ with $\sigma=(a_1,a_2,\ldots,a_s)$, $gcd(\sigma)=d \geq 2$ and $q=md+t$ where $0 \leq t \leq d-1$.  let $H_m=H(n,r,md \mid \sigma)$ be the $\sigma$-hypergraph obtained from $H$ by deleting the top $t$ rows of the grid $V(H)$.  Let $H^*=H(n,\frac{r}{d},m \mid \sigma^*)$ be a $\sigma^*$-hypergraph where $\sigma^*$ is a partition of $\frac{r}{d}$ such that $\sigma^*=(\frac{a_1}{d},\frac{a_2}{d},\ldots,\frac{a_s}{d})$.  Then,
\begin{enumerate}
\item{There is a matching $M^*$ in $H^*$ corresponding to every maximum matching $M$ in $H$.}
\item{There is a matching $M$ in $H$ corresponding to every maximum matching $M^*$ in $H^*$.}
\item{Hence, $\nu(H)=\nu(H_m)=\nu(H^*)$.}
\end{enumerate}
\end{lemma}

\begin{proof}
Let  $M^*$ be a maximum matching in $H^*$, with cardinality $|M^*|$. We ``expand" every vertex in $H^*$ by replacing it with $d$ consecutive vertices. This gives a new $\sigma$-hypergraph $H_m=H(n,r,md \mid \sigma)$, and $M^*$ becomes a matching $M$ in $H_m$ with $|M| = |M^*|$, hence clearly $\nu(H) \geq \nu(H_m) \geq \nu(H^*)$.
 
Now consider $M$, a maximum matching in $H$ with cardinality $|M|$.  By Lemma \ref{not_r_good}, this matching leaves at least $t$ unmatched vertices in each class. By Lemma \ref{consecutive}, there exists another maximum matching $M_c$ in which all vertices in every part of every edge in $M_c$ are consecutive in their respective classes, and in which the unmatched vertices are all consecutive at the top of their respective class.
 
So the top $t$ vertices remain unmatched and hence $M_c$ is also a maximum matching in $H_m$ and now replacing each $d$ consecutive vertices of $H_m$ by a single vertex, we get $H^*=H(n,\frac{r}{d},m \mid \sigma^*)$ with a corresponding matching $M^*$ such that $|M^*| = |M_c| = |M|$.  Hence $\nu(H)=\nu(H_m) \leq \nu(H^*)$.

Therefore $\nu(H)=\nu(H_m)=\nu(H^*)$.
\end{proof}

\subsection{Rectangular Partitions}

We define $\sigma$ to be a rectangular partition if all of its parts are equal.  In this study we shall consider matchings of $\sigma$-hypergraphs where $\sigma$ is rectangular.  In view of Lemma \ref{not_r_good}, we shall start with the rectangular partition all of whose parts are equal to 1.

\begin{lemma} \label{r}
Let  $H=H(n,r,q \mid \sigma)$, where $\sigma=(1,1,\ldots,1)$, and assume $n \geq (r+1)^2$ and $q \geq r$.   Then there is a maximum matching  in which the number of vertices left unmatched is exactly $p$, which is the value of $nq \pmod r $, and hence \[  \nu(H) =  \left \lfloor \frac{nq}{r}  \right \rfloor. \]
\end{lemma}

\begin{proof}
Let $nq = mr + p$, where $0 \leq p \leq r-1$.  What is required is to show that we can find a perfect matching for these $mr$ vertices leaving us only with $p$ unmatched vertices as required.

First, consider any $z \times r$ array or block of vertices of $H$.  Construct edges $E_1,E_2, \ldots, E_z$ where each $E_i$ consists of the $i^{th}$ row of the $z \times r$ array.  This gives a matching which covers this array.

Now consider any $r \times (r+1)$ array.  This time we define the edges $E_1,E_2, \ldots, E_r$ as follows: $E_i$ consists of all the $i^{th}$ row in the $r \times (r+1)$ array except for the $i^{th}$ vertex in that row.  The edge $E_{r+1}$ is then made up of the $r$ vertices which have been left out.  Again $E_1,E_2,\ldots,E_{r+1}$ is a perfect matching of the $r \times (r+1)$ array.  Now we combine these two constructions with Theorem \ref{Frobenius}.  If $n \geq r(r-1)$,  there exist non-negative integers $a$ and $b$ such that $n=ar + b(r+1)$.  Therefore, by splitting up an $r \times n$ grid of vertices into $a$ grids of size $r \times r$  and $b$ grids of size $r \times (r+1)$ , we can use the above constructions to cover the whole $r \times n$ grid with a perfect matching. 

% Since $n \geq (r+1)^2$, there are at least $r$ such grids in an $r \times n$ grid.

Now let us move down the grid of vertices of $H$.  Suppose $q=rk+t$, where $0 \leq t \leq r-1$.  The above construction can be repeated for every one of the $k$ $r \times n$ grids, giving a matching which leaves out the remaining $t \times n$ grid consisting of the bottom $t$ vertices of each class $V_i$  We shall now see how we can cover by a matching many of these vertices.

Starting from the bottom left of the grid of vertices making up $V(H)$, we can cover any $t \times r$ array as we did above in the first construction for a matching of a $z \times r$ grid.  Therefore, if $n=dr+g$, $0 \leq g \leq r-1$, we can cover the bottom $t \times dr$ grid in this way, leaving a $t \times g$ grid unmatched in the bottom right corner of the $q \times n$ grid $V(H)$.  Now, how many of these $tg$ remaining vertices can we cover by a matching?  To see this we need to modify some of the matchings we have constructed so far.

Consider the top grid of size $r \times n$.  In such a grid there are at least $r$ grids of size $r \times r$ or $r \times (r+1)$ since $n \geq (r+1)^2$.   Take the first vertex $v$ of the first edge $E_1$ in the first block and replace it with any vertex $v'$ from the $t \times g$ grid of unmatched vertices, giving us the edge $E_1'=E-v+v'$.  Repeat this by replacing the first vertex in the first edge of the second block with some unmatched vertex in the $t \times g$ array.  Doing this for all the first edges in each of the first $r$ blocks uses $r$ vertices from the $t \times g$ array but creates $r$ unused vertices.  However no two of these vertices are in the same column, therefore they form another new edge.

This procedure can be repeated in various ways, for example:  replacing every second vertex of the same set of edges  with some unused vertex from the $t \times g$ array; or replacing every first vertex of every second edge in a block with an unused vertex from the $t \times g$ array.  We know that  $n \geq ( r+1)^2$.  Hence we have at least $r$ arrays of order $r \times r$ and/or $ r \times (r+1)$ which contain at least $r^2$ edges spread across at least $r$ blocks.  
Replacing  a vertex in the $t \times g$  unmatched grid requires one such edge. We  group the vertices in the $t \times g$ grid into collections of $r$ vertices, and in each such set  we replace the first vertex using an untouched edge from the first block, the second vertex using an untouched edge from the second block and so on.  Since we have less than $(r -1)^2$  vertices in the $t \times g$ grid, and at least $r^2$ edges in the blocks, this process can be carried to the end when less then $r$ vertices remain unmatched.   That is, the number of unmatched vertices is equal to $tg$ reduced $\mod r$.  But $t = q \pmod r$ and $g= n \pmod r$, therefore the number of unused vertices is $nq$ reduced $\mod r$, that is $p$, as required.

\end{proof}

%NOTE: THE CONDITION $n \geq (r+1)^2$ IS NOT REQUIRED FOR THE LAST PART.  PROBABLY $nq \geq (r+1)^2$ IS %SUFFICIENT OR EVEN LESS SINCE THERE ARE SO MANY WAYS TO REPLACE VERTICES IN THE $t \times g$ ARRAY.

We can now use Lemmas \ref{not_r_good}, \ref{expand} and \ref{r} to tackle general rectangular partitions.

\begin{theorem} \label{rectangular}

Let  $H=H(n,r,q \mid \sigma)$, where $\sigma=(\Delta,\Delta,\ldots,\Delta)$, and let $n \geq (r+1)^2$ and $q \geq r\Delta$. Then \[\nu(H)=\left \lfloor \frac{n(q-q( \bmod \Delta))}{r} \right \rfloor.\]
\end{theorem}

\begin{proof}

Let us consider $\Delta \geq 2$, since the case $\Delta=1$ has already been considered in Lemma \ref{r}.  Let $q=m\Delta+t$, where  $t=q \pmod{\Delta}$ so $0 \leq t \leq \Delta-1$ and $m \geq r$.  By Lemma \ref{not_r_good}, a maximum matching will leave at least $t$ vertices unmatched in each class, that is a total of $nt$ vertices.  Hence let us consider then the $\sigma$-hypergraph $H_m=H(n,r,m\Delta \mid \sigma)$ and also let $H^*=H(n,s,m \mid \sigma^*)$, where $\sigma^*=(1,1,\ldots,1)$ and $s=s(\sigma)=\frac{r}{\Delta}$.  By Lemma \ref{expand}, $\nu(H) = \nu(H_m) = \nu(H^*)$.

Now since $n \geq (r+1)^2$  and $m \geq r > s$, we know, by Lemma \ref{r}, that $H^*$ has a maximum matching $M^*$, with $\nu(H^*)= \lfloor \frac{mn}{s} \rfloor$.  Let $mn=fs + z$, where $0 \leq z \leq s-1$, then there are $z$ unmatched vertices in $H^*$, which correspond to $z\Delta$ unmatched vertices in the corresponding matching $M$, in $H_m$,  by Lemma \ref{expand}, where $z=mn \pmod s = \frac{n(q-t)}{\Delta}\pmod s$, and hence $z\Delta=n(q-t) \pmod s$. Now in $H_m$ there are in total $nm\Delta$ vertices, and $nm\Delta=fs\Delta +z \Delta=fr + z\Delta$, hence $fr$ vertices are matched leaving $z\Delta \leq (s-1)\Delta < r$ vertices unmatched, hence $M$ is a maximum matching in $H_m$ and in $H$ by Lemma \ref{expand}.

So, in $H=H(n,r,q \mid \sigma)$, there are $tn+z\Delta$ vertices unmatched and hence
\[|M|=\nu(H)=\frac{nq-nt-z\Delta}{r}=\frac{nq-nt-[(nq-nt)(\bmod s)]}{r} \]
\[=\frac{\lfloor \frac{nq-nt}{s} \rfloor s}{r} = \frac{\lfloor \frac{nq-nt}{s} \rfloor}{\Delta}=\left \lfloor \frac{n(q-q( \bmod \Delta))}{r} \right \rfloor\]

\end{proof}

\subsection{$r$-good partitions}

In view of  Lemma \ref{not_r_good}, we now turn our attention to $\sigma$-hypergraphs in which $\sigma$ is not rectangular but for which $gcd(\sigma) = 1$, and we try to reduce the number of vertices left unmatched.  We shall see below that for such partitions we can get a maximum matching that leaves a relatively small number, in terms of $r$, of unmatched vertices.

Consider $H=H(n,r,q \mid \sigma)$, where $\sigma=(a_1,\ldots,a_s)$.  We call $\sigma$ an \emph{$r$-good partition} if there exists a subsequence $\pi$ of $\sigma$ such that $\sum_{a_j \in \pi}a_j$ is coprime to $r$.  A necessary condition for a partition $\sigma$ to be $r$-good is that  $gcd(\sigma)=1$.  But this is not a sufficient condition as can be seen for $\sigma=(33,45,55,77)$ and $r=210$. 

Let us now consider some important properties of $r$-good partitions.
 
\begin{lemma} \label{rgood}
Let $\sigma=(a_1,a_2,\ldots,a_s)$ be an $r$-good partition of $r$.  Then there exist disjoint sets $A$ and $B$ such that:
\begin{enumerate}
\item{$A \cup B = \{1,2,\ldots,s\}.$}
\item{If \[a=\sum_{j \in A} a_j \mbox{ and } b=\sum_{j \in B} a_j,\] then $gcd(a,b)=gcd(a,r)=gcd(b,r)=1$. }
\item{Let $L=lcm(a,b)$.  Then $gcd(L,r)=1$, and $L \leq \frac{r^2-1}{4}.$}
\end{enumerate}
\end{lemma}
\begin{proof}
Clearly, if we set $A=\{i: a_i \in \pi\}$ and $B=\{i: a_i \not \in \pi\}$, then $A \cup B = \{1,2,\ldots,s\}$, $A$ and $B$ are disjoint, and $a=\sum_{j \in A} a_j$ and $b=\sum_{j \in B} a_j$  are such that $gcd(a,b)=gcd(a,r)=gcd(b,r)=1$. If $L$ is the lowest common multiple of $a$ and $b$ then $gcd(L,r)=1$ by simple number theory.

 Now $L \leq ab$, where $a+b=r$.  The product $ab$ is maximum if $a=b=\frac{r}{2}$ and $a+b$ is even.  But since $a$ and $r$ are coprime and $b$ and $r$ are also coprime, $\min\{a,b\} \leq \frac{r}{2}-1$ and $\max\{a,b\} \geq \frac{r}{2}+1$ and hence $L \leq \frac{r^2-4}{4}$.

If on the otherhand, $r$ is odd,  then the maximum product $ab$ is attained when $\min\{a,b\}=\frac{r-1}{2}$ and $\max\{a,b\}=\frac{r+1}{2}$, giving $L \leq \frac{r^2-1}{4}$.
\end{proof}

Let us now consider a general situation where an $r$-good partition can be used to give a perfect matching in $H(n,r,q \mid \sigma)$.

\begin{lemma} \label {packing}
Consider $H=H(n,r,q \mid \sigma)$ where $\sigma$ is an $r$-good partition with sets $A$ and $B$ and numeric values $a$, $b$ and $L$ as described in Lemma \ref{rgood}.  Then if $r|n$ and $L|q$, $H$ has a perfect matching, that is $\nu(H)=\frac{nq}{r}$.
\end{lemma}

\begin{proof}
 Let $\sigma_a= \{a_j: j \in A\}$ and $\sigma_b= \{a_j: j \in B\}$.  Consider $r$ classes $V_1,V_2,\ldots,V_r$ and take $L$ vertices from each of these classes to form an $L \times r$ grid of vertices.  Let us divide this into two grids of sizes $L \times a$ and $L \times b$.  Consider the $L \times a$ grid:  this can be divided into $\frac{L}{a}$ square grids of size $a \times a$, and we can pack, by Lemma \ref{perfect1}, $a$ sets of $\sigma_a$ into each of these $a \times a$ grids, and a total of $L$ copies of $\sigma_a$. Similarly, the $L \times b$ grid can be divided into $\frac{L}{b}$ square grids of size $b \times b$, and we can pack, by Lemma \ref{perfect1}, $b$ sets of $\sigma_b$ into each of these $b \times b$ grids, and a total of $L$ copies of $\sigma_b$.  Hence each set of vertices for $\sigma_a$ can be matched with a set of vertices of $\sigma_b$ giving an edge of $H$.  Hence we have a perfect matching in an $L \times r$ grid.  If $r|n$ and $L|q$, then we can divide the vertices into a number of $L \times r$ grids of vertices and pack each one as described, giving a perfect matching.

\end{proof}

We now state and prove our main theorem.  We use two classic results in this proof.  Firstly  we use Theorem \ref{Frobenius}, and secondly we use the concept of a Diagonal Latin Square (DLS). A DLS of order $i$ is an $i \times i$ array containing every integer from 0 to $i -1$ in every row, every column and on the leading diagonal.  It is known that there exists a DLS of order $i \times i$, $\forall i \geq 3$.  Using this result we can define a \emph{DLS matching} as follows:

Let $R$ be an $r \times s$ subarray of $V(H)$ for the $\sigma$-hypergraph $H=H(n,r,q \mid \sigma)$ where $\sigma=(a_1,a_2,\ldots,a_s)$.  Let $D$ be a DLS whose entries are $a_1^*,a_2^*,\ldots,a_s^*$.  Corresponding to $D$ we can define a perfect matching of $R$ as follows.  Partition the entries of the $i^{th}$ column of $R$ into parts of size $a_1,a_2,\ldots,a_s$ such that these parts occur in the same order as the corresponding symbols $a_1^*,a_2^*,\ldots,a_s^*$ appear in the $i^{th}$ column of $D$.  Now, taking the parts $a_1,a_2,\ldots,a_s$, one from each column of $R$, gives an edge of $H$ , and these edges together form a perfect matching in $R$ such that each edge corresponds to a row of $D$, and hence the parts corresponding to the $a_i^*$ which run down the main diagonal of $D$ are in different edges.  These parts will be referred to as the parts in the \emph{main diagonal of the DLS matching}.  We call this perfect matching a \emph{DLS matching}.  Figure 1 gives an example of a $9 \times 3$ DLS matching when $\sigma=(4,3,2)$.

\begin{figure}[h!]
\centering
\includegraphics{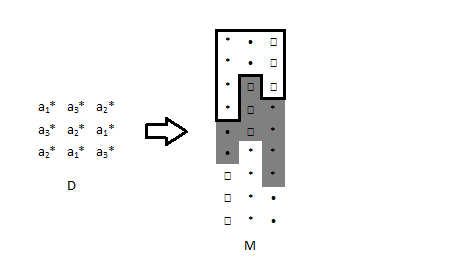}  %or .pdf, especially if you use pdflatex
\caption{The corresponding $10 \times  3$  DLS matching M for $\sigma=(4,3,2)$, where $*$ represents vertices in part $a_1$ of size 4, $\Box$ represents vertices in part $a_2$ of size 3, and \textbullet $\mbox{  }$ represents vertices in part $a_3$ of size 2}
\end{figure}

If $T$ is an $r \times st$ subgrid of $V(H)$,  then we can divide $T$ into $t$ grids of size $r \times s$ and we can construct a DLS matching for each of these subgrids.  We also call this a DLS matching of $T$.

\begin{theorem} \label{perfect2}
Consider $H=H(n,r,q|\sigma)$ where $\sigma$ is an $r$-good partition.  Let $a,b$ and $L$ be as described in Lemma \ref{rgood}.  Then
\begin{enumerate}
\item{If $r|q$ and $n \geq s$,  or $r|n$ and $q \geq (L-1)(r-1)$, then $H$ has a perfect matching, that is $\nu( H)  = \frac{qn}{r}$}
\item{If $q \geq L(r-1)$  and $n \geq s$, then there is a matching in $H$ that leaves at most   $\L(r-1)^2$  vertices unmatched, that is $\nu(H) \geq \frac{qn-L(r-1)^2}{r}$.} 
\item{If $q \geq L(r^2-1)$, $s(\sigma) \geq 3$ and $n \geq s+r$,   then there is a matching in $H$ that leaves at most   $(r-1)^2$ vertices unmatched, that is $\nu(H) \geq \frac{qn -(r-1)^2}{r}$.}
 \end{enumerate}
\end{theorem}

\begin{proof}

\noindent 1. \indent If $r|q$ and $n \geq s$, then by Lemma \ref{perfect1}, $H$ has a perfect matching.

If $r|n$ and $q \geq m = (L-1)(r-1)$, then clearly $gcd(a,b)=gcd(a,r)=gcd(b,r)= gcd(L,r)=1$ by elementary number theory.  Then by Theorem \ref{Frobenius}, if $q\geq (L-1)(r-1)$,  there exist non-negative integers $x$ and $y$ such that $q=xL+yr$.  Hence we can separate the $nq$ vertices into two grids: one which is $xL \times n$, and one which is $yr \times n$.  The former grid has a perfect matching by Lemma \ref{packing}, while the latter has a perfect matching by Lemma \ref{perfect1}.  Hence $H$ has a perfect matching if $q \geq m=(L-1)(r-1)$
\bigskip

\noindent 2. \indent Suppose now that $q \geq m+r-1$ and $n \geq s$.  Consider the $q \times n$ grid of vertices.  For each $r \times n$ grid there is a perfect matching by Lemma \ref{perfect1}.  We take as many such $r \times n$ grids as possible, as long as the left over grid is at least $m \times n = (L-1)(r-1) \times n$.  So, when this process of packing $r \times n$ grids is stopped we are left with a $q_{\scriptscriptstyle1} \times n$ grid of unmatched vertices, for some integer $q_{\scriptscriptstyle1}$ such that $(L-1)(r-1) \leq q_{\scriptscriptstyle1} \leq m +(r-1)=(L-1)(r-1)+(r-1)=L(r-1)$ (otherwise we can pack one more strip). Now we use the fact  that $q_{\scriptscriptstyle1} \geq m$ and hence it  is in the range  where the  $q_{\scriptscriptstyle1} \times n$ subgrid has a perfect matching if $r |n$ . Let $n = tr +b$  such that $0 \leq b \leq  r-1$.  Then  the $q_{\scriptscriptstyle1} \times tr$ grid has a perfect matching, leaving $bq_{\scriptscriptstyle1}$ vertices unmatched.  But $b \leq r-1$ and $q_{\scriptscriptstyle1} \leq L(r-1)$, hence $bq_{\scriptscriptstyle1} \leq L(r-1)^2$.
\bigskip

\noindent 3. \indent Suppose now  that $q \geq L(r^2-1)=L(r-1)(r+1)=(m+r-1)(r+1) \geq m+r-1=(L-1)(r-1)+(r-1)=L(r-1)$ and $n \geq s+r$.  Let $f$ be the largest integer such that $n = fs + h$, $h \geq  r$.   Clearly $f  \geq 1$.  For each $r \times n$ subgrid of $V(H)$  there is a perfect matching by Lemma \ref{perfect1}.  We construct this matching such that it forms a DLS matching for the $r \times fs$ subgrid made up of $f$ grids of size $r \times s$.   The remaining $r \times h$ part, where $r \leq h \leq (r+s-1)$ is given any perfect matching which is possible by Lemma \ref{perfect1}. 

Starting with the first $r$ rows of $V(H)$, we take as many such $r \times n$ grids as long as the left over grid is at least $m \times n = (L-1)(r-1) \times n$.  Let the remaining grid be $q_{\scriptscriptstyle1} \times n$ --- then $ m = (L-1)(r-1) \leq q_{\scriptscriptstyle1} \leq m  +(r-1)  = (L- 1)(r-1) +( r-1) = L(r-1)$ (otherwise we can take one more $r \times n$ grid).   Let $n = tr +b$  such that $0 \leq b \leq  r-1$.  Since $q_{\scriptscriptstyle1} \geq m$, then by part 1 of this theorem, the $q_{\scriptscriptstyle1} \times tr$ grid has a perfect matching, leaving $q_{\scriptscriptstyle1} \times b$ vertices unmatched, where $b \leq r-1$ and $q_{\scriptscriptstyle1} \leq L(r-1)$.   Hence $bq_{\scriptscriptstyle1} \leq L(r-1)^2$.  Now let $q_{\scriptscriptstyle1} =  pr + z$, where $0 \leq z \leq r-1$.  Clearly if $p  = 0$ then $bq_{\scriptscriptstyle1} \leq (r-1)^2$ as stated.  So let us assume $p \geq 1$.  We now show how to match more vertices from this remaining grid.  We first observe  that since $b \leq r-1$,  we have at least $\frac{( n- r +1)}{s} =  f $  DLS matchings of size $r \times s$ whose columns are distinct from those of the unmatched grid.  We take $r$ vertices in a column in the unmatched grid and partition them into parts of size $a_1, a_2, \ldots ,a_s$. We take a DLS and replace a part consisting of $a_1$ vertices from the main diagonal of the matching by the corresponding part from the $r$ vertices in the unmatched columns, and so on.  Every original edge remains a valid one under this exchange.  But now $a_1, \ldots,a_s$ in the original DLS matching are not used.  But by the structure of the DLS matchings they form a valid edge.  Hence every DLS matching  reduces  the number of unmatched vertices by $r$, and if we have sufficiently many DLS matchings, we can match  $pr \times b$ vertices leaving exactly $z \times b$ vertices unmatched. We need exactly $pb$ DLS matchings for this.  Now $pb =  b \frac{(q_{\scriptscriptstyle1} -z)}{r}  < q_{\scriptscriptstyle1}   \leq L(r-1)$ .  So with $f =\frac{n- h}{s}$  and  $g = \frac{q - q_{\scriptscriptstyle1}}{r}$  we  need  $fg \geq  q_{\scriptscriptstyle1}$.  In particular since $q  \geq L(r-1)^2  = (r+1) L(r-1) \geq (r+1)q{\scriptscriptstyle1}$ and $n  \geq r+s$,  we get $f \geq 1$, $g \geq \frac{rq_{\scriptscriptstyle1}}{r} = q_{\scriptscriptstyle1}$ and $fg \geq q_{\scriptscriptstyle1}$, which is the number of DLS matchings required.  Hence we can conclude that   there will be at most $(r-1)^2$ vertices left unmatched, and hence $\nu(H) \geq \frac{qn -(r-1)^2}{r}$.

NOTE: For $s(\sigma)=2$ and $\sigma$ an $r$-good partition, we can work in a very similar way without the use of a DLS matching. Very breifly we proceed as follows. Let $M$ be an $r\times 2$ grid containing two edges in the matching, the first edge made up of the top $a_1$ vertices in the first columns of $M$ and the top $a_2$ vertices in the second column of $M$. The other vertices of $M$ form the second edge of the matching. This $M$ is the analogue of the previous DLS matching. Now let $C$ be an unmatched column of $r$ vertices which, in $V(H)$, is not in any of the two columns containing $M$. The matching can be augmented to include the vertices of $C$: form an edge using the top $a_1$ vertices in the first column of $M$ together with the top $a_2$ vertices in $C$, and another edge using the top $a_2$ vertices of the second column of $M$ together with the bottom $a_1$ vertices of $C$. Thus $r$ unmatched vertices can be matched in this way, and if $n = r+2$ and $q \geq L(r^2-1)$, we get exactly the same resultas for $s(\sigma) \geq 3$.
\end{proof}

\section{Conclusion}

In our earlier work \cite{CaroLauri14, CLZ2,CLZ1} we used $\sigma$-hypergraphs in order to throw more light on colourings of mixed hypergraphs \cite{voloshin02} and, more generally, constrained colourings of hypergraphs. There we found that these hypergraphs were a very flexible tool for investigating parameters such as the upper and lower chromatic numbers and phenomena such as gaps in the chromatic spectrum of constrained colourings. In this paper we have started to investigate the versatility of $\sigma$-hypergraphs in the study of two other classical areas of hypergraph theory, independence and matchings. 

Our first motivation for this work was a very pleasing link between constrained colourings and independence numbers which holds for general uniform hypergraphs. We believe that the relationship between the independence number $\alpha_\beta(H)$ and $\alpha(H)$,  and the upper and lower chromatic numbers $\overline{\chi}_{\alpha,\beta}$ and $\chi_{\alpha,\beta}$ given in Proposition \ref{linkcolind} shows a facet of constrained colourings which has not been investigated before. In Theorem \ref{alpha_k} we see the advantage of the extra structure afforded by $\sigma$-hypergraphs. This structure enabled us to obtain a complete formula for the $k$-independence number of $\sigma$-hypergraphs whose computation is independent of the size of the hypergraph but depends only on the structure of $\sigma$ and, in fact, for fixed $r$, can be computed in $O(1)$ time.

When it comes to the consideration of matchings in a $\sigma$-hypergraph $H$, our results generally depend on elementary number-theoretic relations between the parameters of $H$ and between the parts of $\sigma$. Under some simple conditions on the parameters of $H$ it is easy to show that it has a perfect matching and, under less restrictive conditions but assuming all the parts of $\sigma$ are equal, we were able to compute the exact number of vertices which are left out of a maximum matching and hence the size of such a matching. For more general $\sigma$ we showed that, if the greatest common factor of its parts is at least 2, then most often a perfect matching is not possible. When this greatest common factor is 1, we were able to determine a good approximation for the number of vertices left out of a maximum matching provided the sum of some parts of $\sigma$ is coprime with $r$, the size of the edges of $H$.

As we have seen, the maximum $k$-independence problem for $\sigma$-hypergraphs can be computed exactly, while on the other hand, for the maximum matching problem for $\sigma$-hypergraphs we  have complete solution in some cases, and in other cases, we have given a tight approximation for the number of vertices left unmatched.  It seems that improving upon our tight approximations and maybe even getting exact solutions  is  a worthy problem to consider.

In colourings of mixed hypergraphs, some of the strongest and most general results were obtained when the underlying hypergraph had the very regular combinatorial structure of a design, as in \cite{colburn99,Gionfriddo04,gionfriddo04bicolouring}. The existence of such hypergraphs is, however, usually very restricted. On the other hand, $\sigma$-hypergraphs have a much less restrictive structure. In \cite{CaroLauri14, CLZ2,CLZ1} we showed that, nevertheless, this structure is rich enough to yield interesting general results on colourings of mixed hypergraphs and $(\alpha,\beta)$-constrained colourings. In this paper we hope to have shown that, even in the classical areas of independence and matchings, $\sigma$-hypergraphs provide a sufficiently rich structure upon which to build meaningful general results.      

\bibliographystyle{plain}
\bibliography{CLZ3}

\begin{thebibliography}{10}

\bibitem{bujtas2007colour}
C.~Bujt{\'a}s and Z.~Tuza.
\newblock Color-bounded hypergraphs, {III}: Model comparison.
\newblock {\em Applicable analysis and discrete mathematics}, 1(1):36--55,
  2007.

\bibitem{bujtastuz09}
C.~Bujt{\'a}s and Z.~Tuza.
\newblock Color-bounded hypergraphs, {I}: {G}eneral results.
\newblock {\em Discrete Mathematics}, 309(15):4890--4902, 2009.

\bibitem{bujtas2009color}
C.~Bujt{\'a}s and Z.~Tuza.
\newblock Color-bounded hypergraphs, {II}: Interval hypergraphs and hypertrees.
\newblock {\em Discrete Mathematics}, 309(22):6391--6401, 2009.

\bibitem{bujtas2010color}
C.~Bujt{\'a}s and Z.~Tuza.
\newblock Color-bounded hypergraphs, {IV}: Stable colorings of hypertrees.
\newblock {\em Discrete Mathematics}, 310(9):1463--1474, 2010.

\bibitem{bujtastuz13}
C.~Bujt{\'a}s and Z.~Tuza.
\newblock Color-bounded hypergraphs, {VI}: Structural and functional jumps in
  complexity.
\newblock {\em Discrete Mathematics}, 313(19):1965--1977, 2013.

\bibitem{bujtasv2011color}
C.~Bujt{\'a}s, Z.~Tuza, and Voloshin V.I.
\newblock Color-bounded hypergraphs,{V}: host graphs and subdivisions.
\newblock {\em Discussiones Mathematicae Graph Theory}, 31(2):223--238, 2011.

\bibitem{CaroHan13}
Y.~Caro and A.~Hansberg.
\newblock New approach to the k-independence number of a graph.
\newblock {\em Electr. J. Comb.}, 20(1):P33, 2013.

\bibitem{CaroLauri14}
Y.~Caro and J.~Lauri.
\newblock Non-monochromatic non-rainbow colourings of $\sigma$-hypergraphs.
\newblock {\em Discrete Mathematics}, 318(0):96 -- 104, 2014.

\bibitem{CLZ2}
Y.~Caro, J.~Lauri, and C.~Zarb.
\newblock $(2,2)$-colourings and clique-free $\sigma$-hypergraphs.
\newblock 2014.
\newblock submitted.

\bibitem{CLZ1}
Y.~Caro, J.~Lauri, and C.~Zarb.
\newblock Constrained colouring and $\sigma$-hypergraphs.
\newblock 2014.
\newblock submitted.

\bibitem{ChellaliFHV12}
M.~Chellali, O.~Favaron, A.~Hansberg, and L.~Volkmann.
\newblock k-domination and k-independence in graphs: A survey.
\newblock {\em Graphs and Combinatorics}, 28(1):1--55, 2012.

\bibitem{colburn99}
C.J. Colburn, J.H. Dinitz, and A.~Rosa.
\newblock Bicolouring {S}teiner triple systems.
\newblock {\em Electronic J. Combin.}, 6:R25, 1999.

\bibitem{garey1979computers}
M.~R. Garey and D.~S. Johnson.
\newblock {\em Computers and intractability}, volume 174.
\newblock freeman San Francisco, 1979.

\bibitem{Gionfriddo04}
L.~Gionfriddo.
\newblock Voloshin's colourings of ${P}_3$-designs.
\newblock {\em Discrete Mathematics}, 275(1–3):137 -- 149, 2004.

\bibitem{gionfriddo04bicolouring}
M.~Gionfriddo, L.~Milazzo, A.~Rosa, and V.~Voloshin.
\newblock Bicolouring steiner systems $s(2,4,v)$.
\newblock {\em Discrete mathematics}, 283(1):249--253, 2004.

\bibitem{HansbergPep13}
A.~Hansberg and R.~Pepper.
\newblock On {\it k}k-domination and {\it j}j-independence in graphs.
\newblock {\em Discrete Applied Mathematics}, 161(10-11):1472--1480, 2013.

\bibitem{hardy1952inequalities}
G.~H. Hardy, J.~E. Littlewood, and G.~P{\'o}lya.
\newblock {\em Inequalities}.
\newblock Cambridge university press, 1952.

\bibitem{huang2012size}
H.~Huang, P.~Loh, and B.~Sudakov.
\newblock The size of a hypergraph and its matching number.
\newblock {\em Combinatorics, Probability and Computing}, 21(03):442--450,
  2012.

\bibitem{plummermatching}
M.~D. Plummer and L.~Lov{\'a}sz.
\newblock {\em Matching theory}.
\newblock Elsevier, 1986.

\bibitem{tuza2008problems}
Z.~Tuza and V.~Voloshin.
\newblock Problems and results on colorings of mixed hypergraphs.
\newblock In {\em Horizons of Combinatorics}, pages 235--255. Springer, 2008.

\bibitem{voloshin02}
V.~I. Voloshin.
\newblock {\em Coloring mixed hypergraphs: theory, algorithms and
  applications}, volume~17 of {\em Fields Institute Monograph}.
\newblock American Mathematical Society, 2002.

\end{thebibliography}

\end{document}